\documentclass[11pt,a4,fleqn]{article}
\usepackage{graphicx}
\usepackage{amsmath,amssymb,latexsym,graphics,epsfig}
\usepackage{hyperref}
\usepackage{color}
\usepackage{amsthm}

\newtheorem{theorem}{\bf Theorem}[section]

\newtheorem{lemma}[theorem]{\bf Lemma}

\newtheorem{emp}[theorem]{\bf Claim }

\newtheorem{conjecture}[theorem]{\bf Conjecture}

\newcommand{\G}{\Gamma}

\numberwithin{equation}{section}

\begin{document}
\title{{\Large Minimum edge cuts of distance-regular and  strongly regular digraphs}}
\author{\small  S. Ashkboos$^{\textrm{a}}$, G.R. Omidi$^{\textrm{b},\textrm{c},1}$, F. Shafiei$^{\textrm{b}}$, K. Tajbakhsh$^{\textrm{d}}$\\
\small  $^{\textrm{a}}$Department of Electrical and Computer Engineering,
Isfahan University
of Technology,\\ \small Isfahan, 84156-83111, Iran\\
\small  $^{\textrm{b}}$Department of Mathematical Sciences,
Isfahan University
of Technology,\\ \small Isfahan, 84156-83111, Iran\\
\small  $^{\textrm{c}}$School of Mathematics, Institute for
Research
in Fundamental Sciences (IPM),\\ \small  P.O.Box: 19395-5746, Tehran, Iran\\
\small $^{\textrm{d}}$Department of Mathematics, Faculty of Mathematical Sciences, Tarbiat Modares
University,\\ \small  Tehran, 14115-134, Iran\\
\small \texttt{E-mails: s.ashkboos@ec.iut.ac.ir, romidi@cc.iut.ac.ir,}\\
\small \texttt{fateme.shafiei@math.iut.ac.ir, khtajbakhsh@modares.ac.ir}}
\date {}

\maketitle \footnotetext[1] {This research is partially
carried out in the IPM-Isfahan Branch and in part supported
by a grant from IPM (No. 94050217).} \vspace*{-0.5cm}

\begin{abstract}
\noindent


In this paper, we show that the edge connectivity of a distance-regular digraph $\G$ with valency $k$  is $k$ and for $k>2$, any minimum edge cut of $\G$  is the set of all edges going into (or coming out of) a single vertex.  Moreover we show that the same result holds for  strongly regular digraphs. These results extend the same known results for undirected case with quite different proofs.


\end{abstract}

\section{Introduction}

A {\it digraph} (or a {\it directed graph}) is an ordered pair  $\G=(V,E)$, where $V$ is a set whose elements are called {\it vertices} or {\it nodes}, and $E$ is a set of ordered pairs of vertices, called {\it arcs} or {\it edges}. In contrast, a graph in which the edges are bidirectional is called an {\it undirected graph}.
A digraph with no multiple edges or loops (corresponding to a binary adjacency matrix with 0's on the diagonal) is called {\it simple}. Here we only consider finite simple graphs and digraphs. A digraph $\G$ is called {\it regular} with degree (valency) $k$, if the in-degree and the out-degree at
each vertex of $\G$ are equal to $k$. We will denote by $\partial_{\G}(x,y)$ (or briefly $\partial(x,y)$) the distance from a vertex $x$ to a vertex $y$ in a digraph $\G$. For every vertex $x$, we define the   {\it directed shell}  $\G^{+}_k (x)$ (resp. $\G^{-}_k (x)$) to be the set of vertices at distance $k$ from $x$ (resp. the set of vertices from which $x$ is at distance $k$). The maximum (directed) distance between distinct pairs of vertices is called the {\it diameter} of $\G$ and is denoted by $D$. The {\it girth $g$} is the smallest length of a cycle in $\G$.
In this paper, by a {\it walk}, {\it path} or {\it cycle}, we mean a directed walk, path or cycle. A digraph is {\it (strongly) connected} if there is a path between every pair of vertices. For a connected digraph $\G$, a set of edges $F\subseteq E(\G)$ (resp. a set of vertices $F\subseteq V(\G)$) is called an {\it edge-cut} (resp. a {\it vertex-cut}) if $\G-F$ is disconnected. The sizes of the minimum edge-cut and the minimum vertex-cut of a connected digraph $\G$ are called the {\it edge connectivity} and the {\it vertex connectivity} of $\G$, respectively. For distinct vertices $u$ and $v$ of $\G$, we say that $u$ is {\it adjacent} to $v$ if there is an edge (directed edge) from $u$ to $v$. For every $x\in V(\G)$ and $A\subseteq V(\G)$,  by $d^{+}_{A}(x)$ and $d^{-}_{A}(x)$ we mean the number of out-neighbors and in-neighbors of $x$ in $A$, respectively. For more information on digraphs, we refer the reader to \cite{BJG}. Throughout this paper, let $\G=(V,E)$ be a connected simple digraph of order $n$ and diameter $D$.\\

A {\it distance-regular graph} is a regular graph such that for any two vertices $v$ and $w$ at distance $i$, the number of vertices adjacent to $w$ and at distance $j$ from $v$ only depends on $i$ and $j$. Distance-regular graphs with diameter two are precisely the strongly regular graphs, which have been studied by several mathematicians \cite{B}. For more background on different concepts of distance-regularity in graphs see \cite{BCN,BH0,CDS,F}. The concept of {\it ``distance-regular digraphs''} was introduced by Damerell \cite{D}. A digraph $\G$ with
diameter $D$ is distance-regular if for every two vertices $u$ and $v$ with $\partial(u,v)=k$ for
$1\leq k\leq D$, the numbers $a_{i1}^{k}=|\G^{+}_i(u)\cap\G^{+}_1(v)|$
for each $i$ with $0\leq i\leq k+1$, are independent of the choices of $u$ and $v$. Trivial examples of distance-regular digraphs are the directed cycles (the distance-regular digraphs of degree 1). Moreover distance-regular digraphs with girth $g=2$ are precisely the distance-regular graphs.  We refer the reader to \cite{EM,LM} and the references therein, for more information on the distance-regular digraphs.\\

If we replace $\G^{+}_1(v)$ by
$\G^{-}_1(v)$, in the definition of distance-regularity, we get a new family of digraphs called {\it ``weakly
distance-regular digraphs''}. This concept was introduced by F. Comellas et al. \cite{CFGM}, as a generalization of distance-regular digraphs. In fact, distance-regular digraphs are precisely weakly distance-regular digraphs with {\it normal} adjacency matrices (a matrix $\textbf{A}$ is normal if $\textbf{A}\textbf{A}^t=\textbf{A}^t\textbf{A}$, where $\textbf{A}^t$ is the transpose of $\textbf{A}$). Also, in \cite{CFGM} it has been shown that a digraph $\G$ of diameter $D$ is weakly distance-regular if, for each nonnegative integer  $\ell\leq D$, the number of walks of length $\ell$ from a vertex $u$ to a vertex $v$ only depends on $\ell$ and their distance $\partial(u,v)$. Note that in \cite{SW}, Suzuki and Wang
suggested that a ``weakly distance-regular digraph'' is a digraph with the following property: for all vertices
$u$ and $v$ with $(\partial(u,v),\partial(v,u))=(k_1,k_2)$, the number of vertices $w$ satisfying
$(\partial(u,w),\partial(w,u))=(i_1,i_2)$ and $(\partial(v,w),\partial(w,v))=(j_1,j_2)$ depends only
on the values $k_1,k_2,i_1,i_2,j_1,j_2$.  In this paper, we do not assume the Suzuki and Wang's definition of weakly distance-regular digraphs and we stress that we only consider the mentioned definition that was introduced by F. Comellas et al. in \cite{CFGM}.\\

The weakly distance-regular digraphs with diameter two are the same as the strongly regular digraphs
introduced by Duval in \cite{Du} as an extension of strongly regular graphs to the directed
case. A $k$-regular digraph on $n$ vertices is called a {\it strongly regular digraph} with parameters $(n,k,t,\lambda,\mu)$ if the number of walks of length two between two
vertices is $t$, $\lambda$ or $\mu$ when these vertices are the same, adjacent, or not adjacent, respectively. The case $t=k$ is the undirected case. On the other extreme, the case $t=0$, we have tournaments. For more details, we refer to Brouwer's website \cite{Br}.\\

In \cite{BM} Brouwer and Mesner showed that the minimum vertex cuts of a given strongly regular graph $\G$ are the sets $\G_1(x)$ of all neighbors of a vertex $x$. In 2005 Brouwer and Haemers proved that a distance-regular graph of degree $k$ can not be disconnected by removing fewer than $k$ edges and for $k>2$ the only disconnecting sets of $k$ edges are the sets of $k$ edges on a single vertex (see \cite{BH}). The same result for the minimum vertex cuts of distance-regular graphs is obtained by  Brouwer and Koolen in \cite{BK}. In fact they showed that the vertex-connectivity of a non-complete distance-regular graph $\G$ of degree $k$ equals $k$ and when $k>2$, the only disconnecting sets of vertices of size not more than $k$ are the point neighborhoods. The eigenvalue methods are the main tools to obtain the most of the
above results. In this paper, we investigate to the minimum edge cuts in distance-regular and strongly regular digraphs and we only use the combinatorial techniques to extend the mentioned results on the minimum edge cuts for directed case.\\

The paper is organized as follows. In the next section, we show that  the edge connectivity of a distance-regular digraph $\G$ with valency $k$  is $k$ and if $\G$ is not an undirected cycle, then any minimum edge cut of $\G$  is the set of all edges going into (or coming out of) a single vertex. In Section 3, we show that the same result holds for strongly regular digraphs. In fact we prove a strongly regular digraph $\G$ with valency $k$ can not be disconnected by removing fewer than $k$ edges and the only disconnecting sets of $k$ edges are the sets of $k$ edges going into (or coming out of) a single vertex, unless $\G$ is either an undirected cycle with four or five vertices or the strongly regular digraph with parameters $(n,k,t,\lambda,\mu)=(6,2,1,0,1)$. Note that distance-regular digraphs are precisely weakly distance-regular digraphs with normal adjacency matrices. Also, weakly distance-regular digraphs with diameter 2 are precisely strongly regular digraphs. Based on the above results, in the last section we conjecture that the same result holds for all weakly distance-regular digraphs. Finally, we give an example that shows that the same result does not hold for the minimum vertex cuts and vertex connectivities of strongly regular digraphs (and so for weakly distance-regular digraphs).


\section{Distance-regular digraphs}

In this section, we investigate to the minimum edge cuts of distance-regular digraphs and we show that the edge connectivity of a distance-regular digraph $\G$ with valency $k$  is $k$ and if $\G$ is not an undirected cycle, then any minimum edge cut of $\G$  is the set of all edges going into (or coming out of) a single vertex. Note that distance-regular digraphs with girth $g=2$ are precisely distance-regular graphs and due to a result of Brouwer and Haemers in \cite{BH}, the edge connectivity of a distance-regular graph $\G$ of degree $k$ equals $k$ and if $k>2$ the minimum edge cuts of $\G$ are the sets of all edges crossing a single vertex. On the other hand, distance-regular graphs with degree 2 are precisely the (undirected) cycles.  Hence here we only focus on  distance-regular digraphs with girth $g\geq 3$. First we give a useful known result that will be used later on.\\

\begin{lemma}\label{ec} (\cite{HMSSY}) In any edge cut $(A,V-A)$ of a regular digraph, the number of edges from $A$ to $V-A$ equals the number of edges from $V-A$ to $A$.
\end{lemma}

We remind that for every two vertices $u$ and $v$ with
$1\leq \partial(u,v)=k\leq D$ of a distance-regular digraph $\G$ with diameter $D$, the numbers $a_{i1}^{k}=|\G^{+}_i(u)\cap\G^{+}_1(v)|$
for each $i$ with $0\leq i\leq k+1$, are independent of the choices of $u$ and $v$. Since the adjacency matrix $\textbf{A}$ of a distance-regular digraph $\G$ is normal, that is, the matrix satisfying $\textbf{A}\textbf{A}^t=\textbf{A}^t\textbf{A}$, we have
$a_{11}^{k}=|\G^{+}_1(u)\cap\G^{+}_1(v)|=|\G^{-}_1(u)\cap\G^{-}_1(v)|$ for two vertices $u$ and $v$ with
$1\leq \partial(u,v)=k\leq D$. Here we denote $a_{11}^{1}$ by $\lambda$.

Now we introduce a family $\mathcal{D}$ of distance-regular digraphs as an extension of trivial examples (the directed cycles). Assume that $A\times B=\{(a,b)| a\in A, b\in B\}$ for two sets $A$ and $B$. For $t\geq 2$, we denote by $C[X_1,X_2,\ldots,X_t]$ a digraph with vertex set $V=\bigcup_{i=1}^{t} X_i$ and edge set $E=\bigcup_{i=1}^{t-1} (X_i\times X_{i+1})\bigcup (X_t\times X_{1})$. If $t\geq 3$  and $\rho=|X_{i}|$ is constant, then $\G=C[X_1,X_2,\ldots,X_t]$ is a distance-regular digraph with $\lambda=0$. We denote this family of distance-regular digraphs  with $\lambda=0$ by $\mathcal{D}$. In the following we will see that the family $\mathcal{D}$ are exactly all distance-regular digraphs  with $\lambda=0$.

In \cite{D} Damerell showed that every distance-regular digraph $\G$ with girth $g$ is stable; that is, $\partial(x,y)+\partial(y,x)=g$
for every two vertices $x$ and $y$ at distance $0<\partial(x,y)<g$. Consequently, every distance-regular digraph  with girth $g\geq 3$ has diameter $D=g$ (long type) or $D=g-1$ (short type). Also, he showed that every distance-regular digraph of long type is obtained
from a distance-regular digraph of short type by a known construction as follows:

Let $\Omega$ be a distance-regular digraph of short type and $m>1$ be an integer. Now let $\G$ be a digraph, where
$$V(\G)=V(\Omega)\times \{1,2,\ldots,m\}$$ and $$E(\G)=\{(u,i)(v,j)|uv\in E(\Omega), 1\leq i,j \leq m\}.$$
It is easy to see that $\G$ is a distance-regular digraph of long type with the same girth as $\Omega$. As you see in the following theorem, Damerell showed that the converse is true.

\begin{theorem}\label{lt-st} (\cite{D}) Every distance-regular digraph $\G$ of long type is obtained from a distance-regular digraph $\Omega$ of short type, of the same girth, by the construction described above. Starting from a distance-regular digraph $\G$ of long type, a distance-regular digraph $\Omega$ of short type is obtained by identifying all antipodal vertices of $\G$.
\end{theorem}

In \cite{LN} it is shown that for every non-trivial distance-regular digraph of short type we have $\lambda>0$. Therefore, the only distance-regular digraphs of short type with $\lambda=0$ are the directed cycles. Now let $\G$ be the distance-regular digraph of long type that is obtained from a distance-regular digraph $\Omega$ of short type by the Damerell's construction described above. Clearly the parameter $\lambda$ for $\G$ is $m$ times of the same parameter for $\Omega$. Therefore, every distance-regular digraph of long type with $\lambda=0$ is obtained from a directed cycle by the Damerell's construction and thus it is a member of $\mathcal{D}$. Hence $\mathcal{D}$ are precisely  all distance-regular digraphs  with $\lambda=0$.

The statement of the following lemma about $a_{11}^{l}$ was shown in \cite{LN} for non-trivial distance-regular digraphs of short type. The proof is not correct as stated, although the statement remains valid as we demonstrate. Here we give an alternative way to prove this result for all distance-regular digraphs with $\lambda\neq 0$.

\begin{lemma}\label{drdg0}
Let $\G$ be a distance-regular digraph with diameter $D$, girth $g\geq 3$ and $\G\notin \mathcal{D}$. Then for every $2\leq l\leq D$, we have $a_{11}^{l}\geq 1$.
\end{lemma}
\begin{proof}

Note that $\G\notin \mathcal{D}$, so we have $\lambda\neq0$. First let $l=g$ (this case can happen only when $\G$ is a distance-regular digraph of long type). Then consider a path $u_1u_2\ldots u_{g+1}$ of minimum length between two vertices
$u=u_1$ and $v=u_{g+1}$ at distance $\partial(u,v)=g$. Clearly $\partial(u_2,v)=D-1=g-1$. Using the fact that $\G$ is stable (which means that $\partial(x,y)+\partial(y,x)=g$
for every two vertices $x$ and $y$ at distance $0<\partial(x,y)<g$) we have $\partial(v,u_2)=1$ and so $\G_1^{+}(u)\cap \G_1^{+}(v)\neq\emptyset$.  This fact implies that $a_{11}^{l}\geq 1$. Now let $2\leq l\leq g-1$. Assume that $w\in V(\G)$, $u\in \G_1^{+}(w)$ and $H$ is the digraph induced by $\G_1^{+}(w)$.  Clearly $|H_{1}^{-}(u)|=\lambda$ and for each $x\in H_{1}^{-}(u)$ we have $g-1=\partial_{\G}(u,x)\leq \partial_{H}(u,x)$, where $\partial_{\G}(u,x)$ and $\partial_{H}(u,x)$ are the distances from $u$ to $x$ in $\G$ and $H$, respectively. Now let $x\in H_{1}^{-}(u)$ and $P=u_0u_1\ldots u_t$ be a minimum path in $H$ from $u_0=u$ to $u_t=x$. We have $t\geq g-1$, since $g-1=\partial_{\G}(u,x)\leq \partial_{H}(u,x)$. First assume that $\partial_{\G}(u,u_i)=\partial_{H}(u,u_i)$ for each $2\leq i\leq t$. Then $t=g-1$ and $u_l\in \G_{l}^{+}(u)$ for every $2\leq l\leq g-1$. Since $w\in \G_1^{-}(u)\cap \G_1^{-}(u_l)$, we have $\G_1^{+}(u)\cap \G_1^{+}(u_l)\neq\emptyset$, which implies that $a_{11}^{l}\geq 1$, this follows from the fact that the adjacency matrix $A$ of $\G$ is normal (this means that $AA^t=A^tA$, where $A^t$ is the transpose of $A$). Now let $2\leq i_1\leq t$ be the minimum number $i$ with $\partial_{\G}(u,u_i)< \partial_{H}(u,u_i)$. Let $S_{j,k}=\{\partial_{\G}(u,u_i)|j\leq i< k\}$ for every $0\leq j< k\leq t+1$ and $S=\{\partial_{\G}(u,u_i)|0\leq i\leq t\}$. Our goal is to show that $\{0,1,2,\ldots,g-1\}\subseteq S$. Therefore, for every $2\leq l\leq g-1$, we have $\G_{l}^{+}(u)\cap V(P)\neq\emptyset$. On the other hand, for each $v\in \G_{l}^{+}(u)\cap V(P)$ we have $w\in \G_1^{-}(u)\cap \G_1^{-}(v)$ and so $\G_1^{+}(u)\cap \G_1^{+}(v)\neq\emptyset$, which implies that $a_{11}^{l}\geq 1$. To show that $\{0,1,2,\ldots,g-1\}\subseteq S$, consider the integers $i_{0}=0< i_{1}< i_2< \cdots < i_m \leq i_{m+1}=t$ with maximum $m$ such that for each $1\leq j\leq m$, an integer $i_j\in (i_{j-1},t]$ is the minimum number with $\partial_{\G}(u_{i_{j-1}},u_{i_{j}})< \partial_{H}(u_{i_{j-1}},u_{i_{j}})$. Clearly $m\geq 1$.

\begin{emp}\label{c1}
For each $0\leq j\leq m$, we have $\{0,1,2,\ldots,\partial_{\G}(u,u_{i_{j}})+i_{j+1}-i_j-1\}\subseteq S_{0,i_{j+1}}$.
\end{emp}
\begin{proof}
We give a proof for Claim \ref{c1} by induction on $j$. For each $0\leq i< i_1$, we have $\partial_{\G}(u,u_i)=\partial_{H}(u,u_i)=i$ and so $S_{0,i_1}=\{0,1,2,\ldots,i_1-1\}$. Hence our claim holds for $j=0$. Now assume that the statement of Claim \ref{c1} holds for an integer $j=k\leq m-1$, we are going to show that the statement of this claim holds for $j=k+1$. That follows from  the equality $S_{0,i_{k+1}}=S_{0,i_{k}}\cup \{\partial_{\G}(u,u_i)|i_k\leq i< i_{k+1}\}$ and the fact that $\partial_{\G}(u,u_i)=\partial_{\G}(u,u_{i_{k}})+\partial_{H}(u_{i_{k}},u_i)= \partial_{\G}(u,u_{i_{k}})+i-i_k$ for every $i_k\leq i< i_{k+1}$.
\end{proof}

Now using Claim \ref{c1} for $j=m$, we have $\{0,1,2,\ldots,\partial_{\G}(u,u_{i_{m}})+t-i_m-1\}\subseteq S_{0,i_{m+1}}$. On the other hand $g-1=\partial_{\G}(u,x)\leq \partial_{\G}(u,u_{i_{m}})+t-i_m$. Therefore $\{0,1,2,\ldots,g-1\}\subseteq S_{0,i_{m+1}}\cup \{\partial_{\G}(u,x)\}=S$ and we are done.

\end{proof}

The following theorem is the main result of this section.

\begin{theorem}\label{drdg}
The edge connectivity of a distance-regular digraph $\G$ is equal to its valency. Moreover if $\G$ is not an undirected cycle, then any minimum edge cut of $\G$  is the set of all edges going into (or coming out of) a single vertex.
\end{theorem}

\begin{proof}

Clearly each undirected cycle is a distance-regular graph with edge connectivity 2 (equals to its valency). Assume that $\G$ is a distance-regular digraph with valency $k$ and it is not an undirected cycle. If $k=1$, then $\G$ is a directed cycle and clearly any minimum edge cut is an edge.
Now, assume that $k >1$. Suppose that $F=[A,B]$ is a minimum edge cut of $\G$. Since the set of all edges going into (or coming out of) a single vertex is an edge cut, we have $|F|\leq k$.

First let $\G\in \mathcal{D}$ and $\G=C[X_1,X_2,\ldots,X_t]$, where $t\geq 2$ and $|X_i|=k$ for each $1\leq i\leq t$. We are going to show that the minimum edge cut $F=[A,B]$ is the set of all edges going into (or coming out of) a single vertex. For simplicity we denote by  $\mathcal{F}$ the digraph induced by $F=[A,B]$ on the vertices of $\G$. With no loss of generality suppose that there is a vertex $x\in X_1$ with  $|\mathcal{F}_1^{+}(x)| \geq \max\{|\mathcal{F}_1^{+}(z)|,|\mathcal{F}_1^{-}(z)|\}$ for every vertex $z\in V(\G)$. Now let $d=|\mathcal{F}_1^{+}(x)|$. If $d=k$, then $F$ is the set of all edges coming out of $x$ and there is no thing to prove. So assume that $d<k$ and $Y_2\subset X_2$ is the set of all vertices $y$ such that $xy\in F$. Since $F$ is an edge cut there is no path from $x$ to each $y\in Y_2$ in $\G-F$. Since $d<k$, there are vertices $x_2\in X_2, x_3\in X_3,\ldots, x_t\in X_t$ such that $xx_{2}\notin F$ and $x_ix_{i+1}\notin F$ for each $2\leq i\leq t-1$. Now let $Y_1=\{z\in X_1|x_tz\notin F\}$. Clearly $\{x_t\}\times (X_1-Y_1)\subseteq F$ and $|Y_1|\geq k-d$ since $d=|\mathcal{F}_1^{+}(x)| \geq |\mathcal{F}_1^{+}(x_t)|$. On the other hand, $Y_1\times Y_2\subseteq F$, since there is no path from $x$ to each $y\in Y_2$ in $\G-F$. Hence

 $$k+(k-d)(d-1)\leq k+|Y_1|(|Y_2|-1)=k-|Y_1|+|Y_1||Y_2|\leq |F|\leq k.$$

 Therefore $(k-d)(d-1)\leq 0$, which implies that $d=1$. If $|Y_1|\geq 2$, then $|\mathcal{F}_1^{-}(y)| \geq |Y_1|\geq 2$ for each $y\in Y_2$ (because of $Y_1\times Y_2\subseteq F$), a contradiction to the fact that $|\mathcal{F}_1^{+}(x)| \geq \max\{|\mathcal{F}_1^{+}(z)|,|\mathcal{F}_1^{-}(z)|\}$ for every vertex $z\in V(\G)$. Hence $|Y_1|=1$. Since $|Y_1|\geq k-d\geq 1$, we have $k=2$. Therefore $F=\{xy,x_tx_1\}$, where $y\in X_2$ and $X_1=\{x,x_1\}$. Note that $t=2$ implies that $\G$ is an undirected graph and so there is no thing to prove due to a result of Brouwer and Haemers in \cite{BH}. So $t\geq 3$ and hence for $x'\neq x_t$ in $X_t$, the path $P=xx_2x_3\ldots x_{t-1}x'x_1y$ is a path from $x$ to $y$ in $\G-F$, a contradiction.\\

Now let $\G\notin \mathcal{D}$. Note that $g=2$ implies that $\G$ is an undirected graph and so we are done. Hence we may assume that $g\geq 3$. Set $r=\max \{d^{+}_{A}(x) | x \in A\}$, where $d^{+}_{A}(x)=\G_1^{+}(x)\cap A$. Clearly  every vertex $x \in A$ has at least $k-r$ out-neighbors in $B$. It follows that there are at least $(k-r)|A|$ edges from $A$ to $B$ and so, we have $r+1\leq |A| \leq \frac{k}{k-r}$. Therefore $r\in\{0,k-1,k\}$.

If $r=0$, then $|A|=1$ and $F=[A,B]$ is a set of all edges coming out of a single vertex in $A$ and we are done. Now suppose that $r=k-1$. So $|A|=k$ and since $|F|\leq k$, every vertex $x\in A$ has exactly one out-neighbor in $B$ and $k-1$ out-neighbors in $A$. This implies that $g=2$ and so $\G$ is a distance-regular graph. As we mentioned for this case the assertion holds due to a result of Brouwer and Haemers in \cite{BH}. Hence we may assume that $r=k$ and $g\geq 3$.

Since we do not have an undirected edge (note that $g\geq 3$), we have
$$k|A|-\binom {|A|} {2}\leq |F|\leq k.$$
Therefore $|A| =1$ or $|A| \geq 2k$. If $|A|=1$, then $F=[A,B]$ is a set of all edges coming out of a single vertex in $A$ and there is no thing to prove. Hence we assume that $|A| \geq 2k$. Let $X_{1}$ be the set of all vertices $x\in A$, where $\G_1^{+}(x) \subseteq A$ and $X_{2}=A-X_1$. Clearly   $|X_{2}|\leq|F|\leq k$ and so $|X_{1}| \geq k$. Similarly assume that $Y_{1}$ is the set of all vertices $y\in B$, where $\G_1^{+}(y) \subseteq B$ and $Y_{2}=B-Y_1$. With the same argument, we have $|Y_{2}| \leq k$ and $Y_{1} \geq k$.

Now choose two vertices $x\in X_{1}$ and $y \in Y_{1}$. Set $l=\partial(x,y)$. Clearly $l\geq 2$. Since $\G\notin \mathcal{D}$ and $g\geq 3$, using Lemma \ref{drdg0} we have $a_{11}^{l}\geq 1$ and so $\G^{+}_1(x)\cap\G^{+}_1(y)\neq\emptyset$, a contradiction to the fact that $\G_1^{+}(x) \subseteq A$ and $\G_1^{+}(y) \subseteq B$.
\end{proof}

\section{Strongly regular digraphs}

As an immediate consequent of a result of Brouwer and Haemers in \cite{BH}, any minimum edge cut of a given strongly regular graph with valency $k>2$ is a set of all edges crossing a single vertex. Here we show that the same result is correct for the directed case.


\begin{theorem}\label{srdg}
The edge connectivity of a strongly regular digraph $\G$ equals to its valency. Moreover any minimum edge cut of $\G$ is the set of all edges going into (or coming out of) a single vertex, unless $\G$ is either an undirected cycle with four or five vertices or the strongly regular digraph with parameters $(n,k,t,\lambda,\mu)=(6,2,1,0,1)$.
\end{theorem}

\begin{proof}

Assume that $\G$ is a strongly regular digraph with parameters $(n,k,t,\lambda,\mu)$ and $F=[A,B]$ is a minimum edge cut of $\G$. Clearly each of the mentioned digraphs in Theorem \ref{srdg} is a strongly regular digraph with edge connectivity 2 (equals to its valency). Now assume that $\G$ is not an undirected cycle with four or five vertices or the strongly regular digraph with parameters $(n,k,t,\lambda,\mu)=(6,2,1,0,1)$.  If $\mu=0$, then $\G$ is a complete graph and there is no thing to prove. Therefore we may assume that $\mu\geq 1$. Since the set of all edges going into (or coming out of) a single vertex is an edge cut, we have $|F|\leq k$. By Lemma \ref{ec}, we may assume that $|B| \geq |A|$. Let $r= \max \{d^{+}_{A}(x)|x\in A\}$. Each vertex $x \in A$ has at least $k-r$ out-neighbors in $B$, this implies that there are at least $(k-r)|A|$ edges from $A$ to $B$ and so $r+1\leq |A| \leq \frac{k}{k-r}$. If $r=0$, then $|A|=1$ and there is no thing to prove. Now suppose that $r \neq 0$. Therefore $r\in\{k-1,k\}$ and $r+1\leq |A|\leq |B|$.\\

First let $r = k$. Consider a vertex $x \in A$ with $d^{+}_{A}(x)=k$. Since $\partial(x,y) = 2$ for each $y\in B$, we have $\mu|B|$ paths of length 2 from $x$ to $B$ and so $|B|\leq \mu|B|\leq|F|\leq k$, a contradiction to the fact that $k+1 \leq |A | \leq |B|$.\\

Now let $r=k-1$. Hence every vertex $x\in A$ has at least one out-neighbor in $B$. Using $r+1\leq |A| \leq \frac{k}{k-r}$, we have $|A|= k$ and so every vertex in $A$ has exactly one out-neighbor in $B$ and $k-1$ out-neighbors in $A$. This implies that $t\geq k-1$ and $\lambda \geq k-2$. If $\lambda=k-1$, then one can easily see that $\G$ is a complete graph, a contradiction to our assumption that $\mu\geq 1$. Hence we may assume that $\lambda=k-2$. If $t=k$, then $\G$ is a strongly regular graph and for each edge $xy$ with $x\in A$ and $y\in B$, exactly $k-2$ vertices of $N_A(x)$ are adjacent to $y$. Therefore there are at least $k-1$ paths of length two from $x$ to $z$ for each $z\in N_A(x)\cap N_A(y)$ and so $\lambda \geq k-1$, a contradiction to the fact that $\lambda=k-2$. Now let $t=k-1$. Assume that $xy \in E(\G)$ is an edge such that $x \in A$ and $y \in B$. Since $\lambda=k-2$, exactly $k-2$ vertices of $N_A(x)$ are adjacent to $y$. On the other hand since  $t=k-1$, we have $yx \notin E(\G)$ and so $\partial(y,x)=2$. If $\mu > 1$, then there is at least one vertex $z \in N_{A}^{+}(x)$ such that $yz \in E(\G)$. Since there are at least $k-1$ paths of length 2 from $x$ to $z$, we have $\lambda\geq k-1$, a contradiction. Hence we may assume that $\mu=1$.

If $k=1$, then $\Gamma$ is a directed triangle and so there is no thing to prove.
For $k=2$, the digraph induced by $A$ is a directed cycle of length two and there is no undirected edge between $A$ and $B$.  It is easy to see that $\G$ is the strongly regular digraph with parameters $(6,2,1,0,1)$, a contradiction to our assumptions.
Now let $k\geq 3$ and $xy\in F$, where $x\in A$  and $y\in B$. Since $\lambda=k-2$, exactly $k-2$ vertices of $N_A^{+}(x)=A-\{x\}$ are adjacent to $y$ and so there is exactly one vertex $z\in A$ such that $zy\notin F$. On the other hand, there is exactly one vertex $w\in B$ such that $zw\in F$. Again $k-2$ vertices of $N_A^{+}(z)=A-\{z\}$ are adjacent to $w$ and so for only one vertex $z'\in A$ we have $z'w\notin F$. Therefore $(A\times \{y,w\})-\{zy,z'w\}\subseteq F$, and so $|F|\geq 2k-2>k$, a contradiction.
\end{proof}

\section{Concluding remarks and open problems}
The concept of weakly distance-regular digraphs is an extension of two concepts distance-regular digraphs and strongly regular digraphs. Based on the above results, the investigation to the minimum edge cuts in weakly distance-regular digraphs is an interesting problem. In general undirected cycles are the weakly distance-regular digraphs with a  minimum edge cut that is not a vertex out-neighborhoods (or in-neighborhoods). As we mentioned in Sections 3 besides two small undirected cycles, the strongly regular digraph with parameters $(n,k,t,\lambda,\mu)=(6,2,1,0,1)$ is a nice exception in strongly regular digraphs with a  minimum edge cut that is not a vertex (out-in) neighborhoods.  Therefore, besides undirected cycles it is natural to think about the family of  infinite weakly distance-regular digraphs, each has a  minimum edge cut that is not a vertex out-neighborhoods (or in-neighborhoods). Here we show that such a family exists. In fact for every positive integer $n$, we construct a $2$-regular weakly distance-regular digraph $\G_n$ with $2n$ vertices, diameter $D=[n/2]+1$ such that the statement of Theorem \ref{drdg} does not hold for $\G_n$.
To do this, add the edges $v_iu_i$ and $u_iv_i$ for $1\leq i\leq n$ to two disjoint directed cycles $C_1=v_1v_2v_3\ldots v_{n-1}v_{n}v_{1}$ and $C_2=u_{1}u_{n}u_{n-1}\ldots u_{3}u_{2}u_{1}$, to get a 2-regular weakly distance-regular digraph $\G_n$ with the desired properties. Now based on the previous results and the above discussion we pose the following conjecture about weakly distance-regular digraphs:

\begin{conjecture}\label{wdrdg2}
For every weakly distance-regular digraph $\G$ with valency $k$, the edge connectivity  equals to $k$. Moreover if $k>2$, any minimum edge cut  is the set of all edges going into (or coming out of) a single vertex.
\end{conjecture}

As we mentioned in the first section, Brouwer and Koolen in \cite{BK} showed that the vertex-connectivity of a non-complete distance-regular graph of degree $k>2$ equals $k$, and the only disconnecting sets of vertices of size not more than $k$ are the point neighborhoods. The digraph shown in Figure 1 shows that the same result is not correct for strongly regular digraphs (and so for weakly distance-regular digraphs). As you see in Figure 1, this digraph is a strongly regular digraph with parameters $(8, 3, 2 , 1, 1)$ and vertex cut $U=\{u_1,u_4\}$ of size 2 (less than its valency). We could not find such an example for distance-regular digraphs. An interesting research problem in this direction is to deduce whether the statement of Brouwer and Koolen's result is correct for distance-regular digraphs.\\

\begin{figure}[!h]
\centering
\includegraphics[scale=.55]{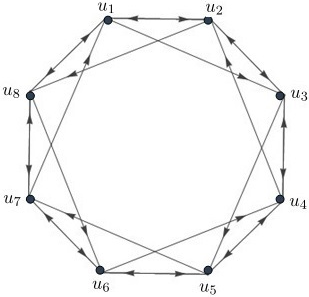}
\caption{A strongly regular digraph with parameters $(8, 3, 2 , 1, 1)$.}
\label{fig01}
\end{figure}

\end{document}